\documentclass[a4paper,12pt]{article}
\usepackage{amsfonts}
\usepackage{latexsym}
\usepackage{amsmath}
\usepackage{MnSymbol}
\usepackage{eufrak} 

\usepackage{indentfirst}
\newcommand{\dotdiv}{\ensuremath{{\ }\mathaccent\cdot{\text{\textendash}}}\ }
\newcommand{\ms}{\scriptscriptstyle} 

\begingroup
\newtheorem{theo}{Theorem}[section]   
\newtheorem{defi}[theo]{Definition}

\newtheorem{propo}[theo]{Proposition}
\newtheorem{coro}[theo]{Corollary}

\newtheorem{remark}[theo]{Remark}
\endgroup

\newenvironment{proof} {{\textbf{Proof.}}} {\ }
                           
\begin{document}

\noindent{\Large{\textbf{About a `concrete' Rauszer Boolean algebra\\
 generated by a preorder}}

\

\noindent \small{LUISA ITURRIOZ,
\textit{Universit\'e de Lyon, Universit\'e Claude Bernard Lyon 1, \\
Institut Camille Jordan, CNRS UMR 5208,
F-69622 Villeurbanne cedex, France.} \\
\noindent E-mail: luisa.iturrioz@math.univ-lyon1.fr}

\

\noindent \textbf{Abstract}

\noindent  Inspired by the fundamental results obtained by P. Halmos and A. Monteiro, concerning equivalence relations and monadic Boolean algebras, we recall the `concrete' Rauszer Boolean algebra pointed out by C. Rauszer (1971), via un preorder $R$. On this algebra we can consider one of the several binary operations defined, in an abstract way, by A. Monteiro (1971).

The Heyting-Brouwer subalgebra of constants (fixpoints), allows us to give a general framework to find representations of several special algebraic structures related to logic.

\

\noindent \textbf{Keywords:}                                                                                                                                                                  

\noindent{\small Rauszer Boolean algebras, monadic Boolean algebras, preorders, $H$-$B$-algebras, representation theorem, three-valued $\L$ukasiewicz algebras, symmetrical Heyting algebras, Nelson algebras, deductive systems, information systems}

\medskip

\section{\hspace{-2.6ex}.{\hspace{1.ex}}Introduction}

In order to approach a set of objects --by excess and by default--, a very old idea is to consider the universe $Ob$ provided with a partition $P$. From an algebraic point of view, this partition generates an \textbf{equivalence relation} $R_{\ms P}$. 
Let $R^{\ast}_{\ms P}$ be the family of all equivalence classes $\vert x \vert$ of $R_{\ms P}$, i.e.\ $R^{\ast}_{\ms P} = \{\vert x \vert : x \in Ob\}$. 

On the Boolean algebra $({\cal P}(Ob), \cap, \cup, -, \emptyset, Ob)$, where ${\cal P}(Ob)$ denotes the powerset of $Ob$, and $\cap, \cup, - , \emptyset, Ob$, are the Boolean operations, the equivalence relation $R_{\ms P}$ induces a \textbf{monadic closure operator} $C_{\ms P}$ and a \textbf{monadic interior operator} $I_{\ms P}$ in the following way, for $A \subseteq Ob$: 
\[
C_{\ms P}A = \bigcup \{\vert x \vert \in R^{\ast}_{\ms P} : x \in A\} ;
\]
\[
I_{\ms P}A = \bigcup\{\vert x \vert \in R^{\ast}_{\ms P} : \vert x \vert \subseteq A\}.
\]

Since $I_{\ms P}A \subseteq A \subseteq C_{\ms P}A$, each subset $A$ of $Ob$ can be approached --by excess and by default-- by the sets $C_{\ms P}A$ and $I_{\ms P}A$. 

Since $R_{\ms P}$ is an equivalence relation, $C_{\ms P}A$ can also be defined by:

\[
C_{\ms P}A = \bigcup \{\vert x \vert \in R^{\ast}_{\ms P} : \vert x \vert \cap A \not = \emptyset \}.
\]

Looking for a general context to manage this type of `concrete' examples we are led to recall very known results \cite{Halmos62}, \cite{Mont60}. 

\

The notion of \textbf{monadic Boolean algebra} was introduced by Halmos \cite{Halmos55} in order to give a systematic algebraic study of the one-variable fragment of the first-order predicate logic \cite{Wajs77}. They are Boolean algebras with, in addition, a unary operator characterized by axioms analogous to those of an \textbf{existential quantifier} $\exists$ or an \textbf{universal quantifier} $\forall$. 

Recall that \cite{Halmos54}, \cite{Halmos62}, \cite{Halmos00} an \textbf{existential quantifier $\exists$} (or a \textbf{monadic operator} $C$, or an S5 operator \cite{Dav54} or a $\nabla$ saturation operator \cite{Mont60}) on a Boolean algebra $(B, \wedge, \vee, -, 0,1)$ is a mapping $\exists : B \rightarrow B$ satisfying the following conditions:

\begin{itemize}
\item[]($\exists0)$ $\exists 0 = 0$
\item[]($\exists1)$ $a \wedge \exists a = a$
\item[]($\exists2)$ $\exists(a \wedge \exists b) = \exists a \wedge \exists b$.
\end{itemize}

The abstract system ${\cal B} = (B, \exists)$ is called a \textbf{monadic Boolean algebra}. For equivalent definitions see \cite{Dav54}, \cite{Rub56}. As usual, the \textbf {universal quantifier} $\forall$ is defined by $\forall x = - \exists -x$.

In  \cite{Halmos55}, \cite{Dav54}, it was shown that the image $\exists(B)$ (i.e.\ the range of the quantifier $\exists$), is a monadic Boolean subalgebra of ${B}$. In addition, $x \in \exists(B)$ if and only if $\exists x = x$, if and only if $\forall x = x$. 
An element $x$ such that $\exists x = x$ (resp. $\forall x = x)$ is called closed (resp. open), constant or a fixpoint, and the set of closed elements is the same as the set of open elements (see for example (\cite{Rub56}, p.31)). In other words, a quantifier $\exists$ is a closure operator on B, for which every open element is closed.

\medskip

Every monadic subalgebra of the `concrete' pair $({\cal P}(Ob), C_{\ms P})$ is called an \textbf{equivalence algebra}. This example of monadic Boolean algebra is typical because a representation theorem (Halmos-Monteiro \cite{Mont60}, \cite{Halmos57}, \cite{Halmos59}) relates the abstract case to this `concrete' structure, as it is recalled in \cite{Itu09}. For the sake of clarity we recall that the operator $C_{\ms P}$(resp.\ $I_{\ms P}$) on ${\cal P}(Ob)$ is complete additive (\cite{Dav54}, p.749), (\cite{Rub56}, p.32), (resp.\ complete multiplicative) which is not necessarily true in a closure algebra.

\medskip

Monadic Boolean algebras appear naturally in several fields \cite{Halmos59}, \cite{Dav54}. In the 1980's, monadic Boolean algebras arose --also naturally-- in the domain of Computer Science because, as Ch.\ Davis wrote, `they do provide generalizations of the simpler and ``more set-theoretical" notion of equivalence relation (\cite{Dav54}, p.748). Let us illustrate this fact by the following example.

\begin{defi} An \textbf {information system} in the sense of Pawlak \cite{Paw91} is a system 
\[(Ob, Att, \{Val_{a} : a \in Att\}, f)\]
where $Ob$ is a nonempty (finite) set called the universe of objects, $Att$ is a nonempty finite set of attributes, each $Val_{a}$ is a nonempty set of values of attribute $a$, and $f$ is a
function $f : Ob \times Att \rightarrow Val$, where $Val = \bigcup_{a \in Att} Val_{a}$.
In this way, for every $x \in Ob$ and $a \in Att$ we have that $f(x,a)= a(x) \in Val_{a}$.\\

An \textbf {equivalence relation} $R_{\ms P}$ on $Ob$, called the \textbf {indiscernibility relation}, can be defined in the following way:

for $x, y \in Ob,\ \ x R_{\ms P} y$ if and only if $f(x,a) = f(y,a)$, for every $a \in Att$.

\noindent The system $(Ob, R_{\ms P})$ is called an \textbf {approximation space.}
\end{defi}

By the construction indicated above, it follows that this equivalence relation $R_{\ms P}$ generates a monadic operator $C_{\ms P}$ and its \textbf{dual} $I_{\ms P}X = \\
- C_{\ms P}-X$ on the Boolean algebra $({\cal P}(Ob), \cap, \cup, -, \emptyset, Ob)$. Thus, the `concrete' structure $({\cal P}(Ob), C_{\ms P})$ is an equivalence algebra.

In this particular context \cite{Paw91}, the sets $I_{\ms P}X$ and $C_{\ms P}X$ are respectively called the \textbf {lower approximation} and the \textbf {upper approximation} of $X$, and a set $X \subseteq Ob$ is called  \textbf{$R_{\ms P}$-definable} if $C_{\ms P}X = I_{\ms P}X$, i.e.\ a constant of the monadic Boolean algebra $({\cal P}(Ob), C_{\ms P})$. Otherwise, $X$ is called a \textbf {rough set}. In the literature, a \textbf {rough set} can also be defined as a pair $[I_{\ms P}X, C_{\ms P}X]$, where $X \subseteq Ob$. 

For a strong relation between rough sets and three-valued {\L}ukasiewiez algebras see \cite{Itu99}. Rough sets, which are pairs of particular Boolean elements as it was exhibited above, provide a general framework to represent three-valued structures \cite{Itu01}, \cite{Itu07}.

\

Moreover, Halmos \cite{Halmos55} proved that any abstract monadic Boolean algebra\\ 
$(A, \wedge, \vee, -, 0, 1, C)$ is semisimple, i.e.\ the intersection of all monadic maximal filters is $\{1\}$. Interested in a Halmos's remark about semisimplicity, Monteiro (\cite{Mont71}, p.419), in an outstanding paper, considered the problem of the semisimplicity in abstract topological Boolean algebras. He established that the monadic Boolean algebras are exactly the topological Boolean algebras which are semisimple.

In the same paper, this author developed the notion of deductive semisimplicity in abstract topological Boolean algebras, in the sense of Tarski's elegant theory of deductive systems. With this purpose in mind, he showed properties of five binary (implication) operations defined on those structures \textbf{in an abstract way}.

\

Interested in logic or applied developments, some authors have replaced the equivalence relation by a preorder. It is the case of C. Rauszer (see \cite{Raus71}, \cite{Raus74}, \cite{Raus77}). 

In the example above a \textbf {preorder} $R$ on $Ob$, called the \textbf {informational inclusion}, can be defined in the following way:

for $x, y \in Ob,\ \ x R y$ if and only if $f(x,a) \subseteq f(y,a)$, for every $a \in Att$.

\

The aim of this paper is to extend the study of lower and upper approximations by means of \textbf{preorders} and to exhibit some properties of the Rauszer Boolean algebra $({\cal P}(Ob), \cap, \cup, -, \emptyset, Ob, I_{\ms R}, C_{\ms R})$ generated by a `concrete' preorder $R$ on the universe $Ob$. Some routine proofs are included for the sake of completeness.

\section{\hspace{-2.6ex}.{\hspace{1.ex}}{A `concrete' Rauszer Boolean algebra }}

\textbf{Rauszer Boolean algebras} are Boolean algebra with, in addition, two particular unary operators, $I$ an interior and $C$ a closure. They were introduced and studied by Rauszer \cite{Raus71} under the name of bi-topological algebras. 
In this section we point out some basic notions related to a `concrete' Rauszer Boolean algebra.

Preorders, which are reflexive and transitive relations, are also named quasiorders or $S4$ relations. It is well known that, each such preorder is associated --in a natural way-- with an equivalence relation (i.e.\ a reflexive, symmetric and transitive relation) and also with orders.

\

Let $Ob$ be a nonempty set (set of objects) and $R$ a preorder relation on $Ob$. For $x \in Ob$, let 
\[
R(x) = \{y \in Ob :  x R y\}
\]
and $R^{\ast} = \{R(x) : x \in Ob\}$.

By the reflexivity of $R$ we infer that $x \in R(x)$. Also, if $z \in R(x)$ and $u \in R(z)$ then $x R z$ and $z R u$, so by transitivity  $x R u$, i.e.\ $u \in R(x)$. Thus $x R z \Leftrightarrow z \in R(x) \Leftrightarrow R(z) \subseteq R(x)$.

\

The converse of $R$, denoted by $S$ is defined by: 
\[
S(x) =  \{y \in Ob :  y R x\}.
\]

As suggested by the above readings, we see that on the Boolean algebra $({\cal P}(Ob), \cap, \cup, -, \emptyset, Ob)$, a preorder relation $R$ induces a unary operator $C_{\ms R}$ in the following way, for $X \subseteq Ob$: 
\[
C_{\ms R}X = \bigcup \{R(x) \in R^{\ast} : x \in X\}.
\]

\begin{propo}\label{axiomsC}
On the Boolean algebra $({\cal P}(Ob), \cap, \cup, -, \emptyset, Ob)$, the mapping $C_{\ms R}: {\cal P}(Ob) \rightarrow {\cal P}(Ob)$ satisfies the following conditions: 

\begin{itemize}
\item[](C1) $C_{\ms R} \emptyset = \emptyset$
\item[](C2) $X \subseteq C_{\ms R}X$
\item[](C3) $C_{\ms R}(X \cup Y) = C_{\ms R}X \cup C_{\ms R}Y$
\item[](C4) $C_{\ms R}(C_{\ms R}X) = C_{\ms R}X$
\end{itemize}
\end{propo}
\begin{proof}
In fact, $\textit{(C1)}$ $C_{\ms R}\emptyset = \bigcup \{R(x) \in R^{\ast} : x \in \emptyset\} = \emptyset$. 
Also, $\textit{(C2)}$ $X \subseteq C_{\ms R}X$ since the relation is reflexive. 

Now we prove that $C_{\ms R}$ is increasing, that is if $X \subseteq Y$, then $C_{\ms R}X \subseteq C_{\ms R}Y$. Let $z \in C_{\ms R}X$ then there is $u \in X \subseteq Y$ such that $z \in R(u)$, thus $z \in C_{\ms R}Y$.

Since $X \subseteq X \cup Y$ and $Y \subseteq X \cup Y$ we have $(i)$ $C_{\ms R}X \cup C_{\ms R}Y \subseteq C_{\ms R}(X \cup Y)$.
On the other hand, let $z \in C_{\ms R}(X \cup Y)$ then there is $u \in X \cup Y$ such that $z \in R(u)$. If $u \in X$ then $z \in C_{\ms R}X$, and if $u \in Y$ then $z \in C_{\ms R}Y$,  i.e.\ $(ii)$ $ z \in C_{\ms R}X \cup C_{\ms R}Y$. From $(i)$ and $(ii)$ we get $\textit{(C3)}$.

From $\textit{(C2)}$ we have $(i)$ $C_{\ms R}X \subseteq C_{\ms R}(C_{\ms R}X)$. On the other hand, let $z \in C_{\ms R}(C_{\ms R}X)$ then there is $(a)$ $u \in C_{\ms R}X$ such that $z \in R(u)$. From $(a)$ there is $(b)$ $v \in X$ such that $u \in R(v)$. These facts mean that $u R z$ and $v R u$. Since $R$ is transitive $v R z$. Hence $z \in R(v)$ and by $(b)$ we get $z \in \bigcup \{R(v) : v \in X\} = C_{\ms R}X$.

The proof is now complete.
\end{proof}

\

From conditions $\textit{(C1)-(C4)}$ we deduce that $C_{\ms R}$ is a \textbf{closure operator}, or a $S4$ operator on the Boolean algebra ${\cal P}(Ob)$. 

\

In ${\cal P}(Ob)$, we can define the operator $I_{\ms R}$, via the preorder $R$, in the following way:
\[
I_{\ms R}X =  \bigcup \{R(x) \in R^{\ast}  \ \ \text{such that} \ \ R(x) \subseteq X\}
\]

\begin{propo}\label{axiomsI}
On the Boolean algebra $({\cal P}(Ob), \cap, \cup, -, \emptyset, Ob)$, the mapping $I_{\ms R}: {\cal P}(Ob) \rightarrow {\cal P}(Ob)$ satisfies the following conditions: 

\begin{itemize}
\item[](I1) $I_{\ms R} Ob = Ob$
\item[](I2) $I_{\ms R}X \subseteq X$
\item[](I3) $I_{\ms R}(X \cap Y) = I_{\ms R}X \cap I_{\ms R}Y$
\item[](I4) $I_{\ms R}(I_{\ms R}X) = I_{\ms R}X$
\end{itemize}
\end{propo}
\begin{proof}
In fact, $\textit{(I1)}$ $I_{\ms R}Ob = \bigcup \{R(x) \in R^{\ast} : x \in R(x) \subseteq Ob\} = Ob$. 
Also, $\textit{(I2)}$ $I_{\ms R}X \subseteq X$ since the relation is reflexive. 

Now we prove that $I_{\ms R}$ is increasing, that is if $X \subseteq Y$, then $I_{\ms R}X \subseteq I_{\ms R}Y$. Let $z \in I_{\ms R}X$, then there is $u \in Ob$ such that $z \in R(u) \subseteq X \subseteq Y$, thus $z \in I_{\ms R}Y$.

Since $X \cap Y \subseteq X$ and $X \cap Y \subseteq Y$ we have $(i)$ $I_{\ms R}(X \cap Y) \subseteq I_{\ms R}X \cap I_{\ms R}Y$.
On the other hand, let $z \in I_{\ms R}X \cap I_{\ms R}Y$, then there are $u, v \in Ob$ such that $z \in R(u) \subseteq X$ and $z \in R(v) \subseteq Y$. Since $R$ is a preorder, $z \in R(z) \subseteq R(u) \subseteq X$ and $z \in R(z) \subseteq R(v) \subseteq Y$. Hence $z \in R(z) \subseteq X \cap Y$, that is $z \in I_{\ms R}(X \cap Y)$. From $(i)$ and $(ii)$ we get $\textit{(I3)}$.

From $\textit{(I2)}$ we have $(i)$ $I_{\ms R}(I_{\ms R}X) \subseteq I_{\ms R}X$. On the other hand, let $z \in I_{\ms R}X$ then there is $R(u) \in R^{\ast}$ such that $z \in R(u) \subseteq X$. Since $R(u) \subseteq X$ we have $R(u) \subseteq I_{\ms R}X$. Thus $z \in R(u) \subseteq I_{\ms R}X$, i.e.\ $z \in I_{\ms R}(I_{\ms R}X)$.

The proof is now complete.
\end{proof}

\

From conditions $\textit{(I1)-(I4)}$ we deduce that $I_{\ms R}$ is an \textbf{interior operator}, or a $S4$ operator on the Boolean algebra ${\cal P}(Ob)$. 



\

The `concrete' Rauszer Boolean algebra ${\cal B}= ({\cal P}(Ob), I_{\ms R}, C_{\ms R})$ has many interesting algebraic properties. 

We take note of the fact that, in the particular case of monadic Boolean algebras, the operators $C_{\ms R}X$ and $I_{\ms R}X$ are related by means of the Boolean negation. But this is not the case here. 

However, we can prove that (\cite{Raus74}, p.228) 
\begin{center}
$I_{\ms R}C_{\ms R}X= C_{\ms R}X$ \quad and \quad  $C_{\ms R}I_{\ms R}X= I_{\ms R}X
$, 
\end{center}
that is, they are \textbf{conjugate} over ${\cal P}(Ob)$. 

\noindent In fact, by $\textit{(I2)}$ we have $I_{\ms R}C_{\ms R}X \subseteq C_{\ms R}X$. To prove the converse, let $z \in C_{\ms R}X$ then there is $u \in X$ such that $z \in R(u) \subseteq C_{\ms R}X$.
Since $R$ is a preorder, 

\noindent $z \in R(z) \subseteq R(u) \subseteq C_{\ms R}X$, that is 
$z \in I_{\ms R}C_{\ms R} X$. The proof of $C_{\ms R}I_{\ms R}X= I_{\ms R}X$ is similar.

\

A set $X \in {\cal P}(Ob)$ is called \textbf{$R$-closed} in the case $C_{\ms R}X = X$ and \textbf{$R$-open} in the case $I_{\ms R}X = X$. As in the case of monadic Boolean algebras we have here that 
\begin{center}
 $X$ is $R$-open \quad if and only if \quad $X$ is $R$-closed.
\end{center}
Indeed, if $X$ is $R$-open we have $I_{\ms R}X = X$. Hence $C_{\ms R}X = C_{\ms R}I_{\ms R}X = I_{\ms R}X = X$. The proof of the converse is analogous.

\

Using $S$, the converse of $R$, we can also consider the operation 

\noindent $I_{\ms S}X = \bigcup \{S(x) : S(x) \subseteq X\}$. In this case we have:
\[
C_{\ms R}X = -I_{\ms S}-X  \ \text{and} \ \ C_{\ms S}X = -I_{\ms R}-X
\]

In fact, let $z \in C_{\ms R}X$, so  there is $u \in X$ such that $z \in R(u)$, i.e.\ $(a)$ $u R z$. 
If $z \in I_{S}-X$, then there is $v \in Ob$ such that $z \in S(v) \subseteq -X$. In this case, $(b)$ $z R v$. 
From $(a)$ and $(b)$ we get $u R v$ by transitivity, so $u \in S(v) \subseteq -X$, a contradiction. 
This proves that $(i)$ $C_{\ms R}X \subseteq -I_{\ms S}-X$.
To prove the converse inclusion, let $z \in -I_{S}-X$. Then $z \not \in I_{S}-X =
\bigcup\{S(y) : S(y) \subseteq -X\}$. Hence there is a $x \in S(z) \cap X$, i.e.\ $x R z$ and $x \in X$. Therefore $z \in R(x)$ and $x \in X$. Thus $x \in C_{\ms R}X$. This shows $(ii)$ $-I_{\ms S}-X \subseteq C_{\ms R}X$

The proof of the other equality is similar.

\

Let ${\cal O}_{R}$ be the family of all \textbf{$R$-open} elements and $I_{\ms R}({\cal P}(Ob))$ the image of ${\cal P}(Ob)$ by $I_{\ms R}$. We have the following equivalences:

$Z \in I_{\ms R}({\cal P}(Ob))$ $\iff$ $\text{there is}\ X \in {\cal P}(Ob)\ \text{such that}\ I_{\ms R}X = Z$ $\iff$\\
 $\text{there is}\ X \in {\cal P}(Ob)\ \text{such that}\ I_{\ms R}Z = I_{\ms R}I_{\ms R}X = I_{\ms R}X = Z$ $\iff$ $Z \in {\cal O}_{R}$

\

Since $I_{\ms R}$ is an interior operator, then the image $I_{\ms R}({\cal P}(Ob)) = ({\cal O}_{R}, \cap, \cup, \emptyset, Ob)$ is a distributive lattice, with zero and unit.

In addition, ${\cal O}_{R}$ satisfies the following property (\cite{Mont42}, p.177):
\begin{itemize}
\item[] If for all $k \in K$, $X_k \in {\cal O}_{R}$ then the lower upper bound (l.u.b.) $\bigvee_{k  \in K} X_k = \bigcup_{k  \in K} X_k$ is in ${\cal O}_{R}$.
\end{itemize}

In other words, this means that $I_{\ms R}(\bigcup_{k \in K} X_k) = \bigcup_{k \in K} X_k$. Hence ${\cal O}_{R}$ is a \textbf{sup-complete lattice}. 

\

In the `concrete' Rauszer Boolean algebra ${\cal B} = ({\cal P}(Ob), I_{\ms R}, C_{\ms R})$, the lattice ${\cal O}_{R}$ has another property which is, in general, not true in topological spaces.

\begin{propo}\label{}
If for all $k \in K$, $X_k \in {\cal O}_{R}$ then $I_{\ms R}(\bigcap_{k \in K} X_k) = \bigcap_{k \in K} X_k$.
\end{propo}
\begin{proof}
Indeed, $(i)$ $I_{\ms R}(\bigcap_{k \in K} X_k) \subseteq \bigcap_{k \in K} X_{k}$ by $\textit{(I2)}$. On the other hand, we will prove that $\bigcap_{k \in K} X_k \subseteq I_{\ms R}(\bigcap_{k \in K}X_k) = \bigcup \{R(u) : R(u) \subseteq \bigcap_{k \in K}X_k\}$.

Let $z \in \bigcap_{k \in K} X_k$. For all $k \in K$ we have $z \in X_k = I_{\ms R}X_k$, that is, for all $k \in K$ there is $u_{k} \in X_{k}$ such that $z \in R(z) \subseteq R(u_{k}) \subseteq X_{k}$. Hence $z \in R(z) \subseteq \bigcap_{k \in K} X_k$. This leads to $z \in I_{\ms R}(\bigcap_{k \in K} X_k)$, that is $(ii)$ $\bigcap_{k \in K} X_k \subseteq I_{\ms R}(\bigcap_{k \in K}X_k)$. 

From $(i)$ and $(ii)$ we get the result.
\end{proof}

Owing to the previous results, we infer that the ordered set $({\cal O}_{R}, \cap, \cup, \emptyset, Ob)$ is a \textbf{complete} lattice with zero and unit. The l.u.b.\ and g.l.b.\ being the intersection and union of sets respectively.

Let ${\cal C}_{R}$ be the family of all \textbf{$R$-closed} elements under $R$ and $C_{\ms R}(Ob)$ the image of $Ob$ by $C_{\ms R}$. By a result above, we conclude that ${\cal C}_{R} = {\cal O}_{R}$.

\ 

Moreover, based on semisimplicity motivations, A. Monteiro \cite{Mont71}, has studied properties of several binary operations in \textbf{abstract} topological Boolean algebras $(A, I)$, where $A$ is a Boolean algebra and $I$ is an interior operator on $A$. In particular, he dealt with an implication $\Rightarrow$ (\cite{Mont71}, p.432), (\cite{Mont80}, p.33), defined by:
\begin{align*}
a \Rightarrow b &= I(Ia \supset Ib)\\
\righthalfcap a &= a \Rightarrow 0.
\end{align*}
\noindent where $\supset$ is the classical implication $x \supset y = - x \cup y$.

\

In view of our construction, this operation on ${\cal O}_{R}$ is defined, for all $G, H \in {\cal O}_{R}$ by
\[
G \Rightarrow H = I_{\ms R}(G \supset H)
\]
On ${\cal O}_{R}$, the operation $\Rightarrow$ is the Heyting implication since it satisfies the following conditions, for all $G, H, X \in {\cal O}_{R}$:
\begin{itemize}
\item[]$(H1)$ $G \cap I_{\ms R}(G \supset H) \subseteq H$
\item[]$(H2)$ if $G \cap X \subseteq H$, then  $X \subseteq I_{\ms R}(G \supset H)$
\end{itemize}

By duality, the connective $\dotdiv$  can be expressed, for all $G, H \in {\cal O}_{R}$, as:
\[
G \dotdiv H = C_{\ms R}(G \cap -H)
\]
On ${\cal O}_{R}$, the operation $\dotdiv$ is the Brouwer implication (also called the pseudo-difference or the residual) since it satisfies the following conditions,\\
for all $G, H, X \in {\cal O}_{R}$ (\cite{War-Dil39}, p.337):
\begin{itemize}
\item[]$(B1)$ $G \subseteq H \cup C_{\ms R}(G \cap -H)$
\item[]$(B2)$ if $G \subseteq H \cup X$, then  $C_{\ms R}(G \cap -H)\subseteq X$
\end{itemize}

Therefore the system $({\cal O}_{R}, \cap, \cup, \Rightarrow, \dotdiv, \righthalfcap, \lefthalfcap, \emptyset, Ob)$ is a Heyting-Brouwer algebra. Here $\lefthalfcap G = Ob \dotdiv G$.

This type of algebras were remarked by McKinsey and Tarski in (\cite{McK-T46}, p.129) and referred to as ``double Brouwerian algebras".  According to these authors, this notion seems to have been discussed for the first time in a paper by Skolem in 1919 (implicative and subtractive lattices). In the 1970's, they were extensively investigated by Rauszer in several papers, under the name of semi-Boolean algebras \cite{Raus71}. We remark that, in the literature, this latter name has also been used for other structures. 
They are an algebraic counterpart of an extension of the intuitionistic logic that she called Heyting-Brouwer ($H$-$B$)-logic. For this reason, we preferred to call them \textbf{Heyting-Brouwer algebras} ($\textbf{H}$-$\textbf{B}$-algebras in brief). For more information see \cite{Itu76a}, \cite{Itu76b}. 

\

We close this section recalling the following structural results, that will be used in the sequel.

\begin{defi}\label{ded-alg}
A \textbf{deductive algebra} $(A, \rightarrowtail, 1)$ is an algebra of type $(2, 0)$ satisfying the following conditions (\cite{Mont80}, p.5):
\begin{itemize}
\item[](I1) $x \rightarrowtail (y \rightarrowtail x) = 1$
\item[](I2) $(x \rightarrowtail (y \rightarrowtail z)) \rightarrowtail ((x \rightarrowtail y) \rightarrowtail (x \rightarrowtail z)) = 1$
\item[](I3)  if $1 \rightarrowtail x = 1$, then $x = 1$
\end{itemize}
\end{defi}

Incidentally, if $(A, \rightarrowtail, 1)$ is a deductive algebra, then $x \rightarrowtail 1 = 1$ (\cite{Mont80}, p.6) and $x \rightarrowtail x = 1$ (\cite{Mont80}, p.11).

\begin{defi}
A subset $D$ of a deductive algebra $(A, \rightarrowtail, 1)$ is said to be a \textbf{deductive system} if:
\begin{itemize}
\item[](D1) $1 \in D$
\item[](D2) if\ $a, a  \rightarrowtail b \in D$, then\ $b \in D$\ (modus ponens)
\end{itemize}
\end{defi}

We note that the systems $({\cal P}(Ob), \supset, Ob)$ and $({\cal P}(Ob), \Rightarrow, Ob)$ are \textbf{deductive algebras} (\cite{Mont80}, p.33) (Deductive algebras are also called quasi-I-algebras by A. Horn (1962)).

If an algebraic system is a deductive algebra, then we can apply the fundamental results proved by A. Monteiro (\cite{Mont71}, pp.427-431) concerning the theory of deductive systems, founded and developed by Tarski. 

This means, for example, that the deduction theorem is satisfied. The deductive system generated by a set $Z \not \equal \emptyset$ is:
\[
D(Z) = \{x \in A : (z_{1} \rightarrowtail (z_{2} \rightarrowtail \ldots \rightarrowtail ( z_{n} \rightarrowtail x) \ldots )) = 1, \text{with}\ z_{1},  z_{2}, \ldots , z_{n} \in Z\}
\]

\noindent and the deductive system $D(D_{1}, a)$ generated by a deductive system $D_{1}$ and a fixed element $a \not \in D_{1}$ is:
\[
D(D_{1},a) = \{x \in A \quad \text{such that} \quad a \rightarrowtail x \in D_{1}\}
\]

\section{\hspace{-2.6ex}.{\hspace{1.ex}}{Representation theorems in an unified form}}

Following the point of view expressed by McKinsey and Tarski (\cite{McK-T46}, p.130) in the domain of Heyting algebras, we are interested to show that the method of constructing $H$-$B$-algebras as above is the most general one, i.e.\ that every Heyting-Brouwer algebra can be embedded in a `concrete' Rauszer Boolean algebra enriched with an abstract binary operation, and more precisely that it can be represented as a subalgebra of the $R$-open (or $R$-closed) elements ${\cal O}_{R}$ of this algebra.

Indeed, we can envisage more, because the advantage of this construction is that it provides a general framework for representation of several known structures. 

\

Looking for representations of a distributive lattice by a field of objects of some sort satisfying the $T_0$ axiom of separability, it is known that (\cite{Bir-Fr48}, p.306) there is no loss of generality if we confine attention to sets of \textbf{prime filters} $Ob$, ordered by inclusion, i.e.\ $R$ is $\subseteq$. Thus if $P \in Ob$ then $R(P) = \{Q \in Ob : P \subseteq Q \}$.

\medskip

For the sake of clarity we recall that a subset $F$ of a lattice $(A, \wedge, \vee, 0, 1)$ is said to be a \textbf{filter} if the following conditions are satisfied:

\medskip

\noindent $(f1)$ $1 \in F$; \quad $(f2)$ if $a, b \in F$, then $a \wedge b \in F$; \quad $(f3)$ if $a \in F$ and $a \leq b$, then $b \in F$;

\medskip

\noindent and a filter $P$ is said to be \textbf{prime} if it satisfies the conditions:

\medskip

$(p1) $$P$ is proper, that is $P \not = A$; \quad $(p2)$ if\ $a \vee b \in P$\ implies\ $a \in P$ or $b \in P$

\

We note, incidentally, that for Heyting algebras, the kernel of a homomorphism from a Heyting algebra into another, is a filter. Also, the notions of deductive systems and filters are equivalent (A. Monteiro, 1959). 

\

In the remainder of this paper, we are going to illustrate how the construction in Section \textbf{2.} gives a general support to represent algebraic structures as: $H$-$B$-algebras, three-valued {\L}ukasiewicz algebras, symmetrical Heyting algebras and Nelson ones.

In the sequel, we assume some familiarity with these structures.

\ 

\noindent \textbf{A) Representation of $H$-$B$-algebras and three-valued {\L}ukasiewicz algebras}

Let $(A, \wedge, \vee, \Rightarrow, \dotdiv, \righthalfcap, \lefthalfcap, 0, 1)$ be a $H$-$B$-algebra. 

\begin{theo}\label{repr-H-B}
For every $H$-$B$-algebra $A$, there exists an isomorphism $h$ from $A$ into the $H$-$B$-algebra of sets ${\cal O}_{R}$ of the `concrete' Rauszer Boolean algebra derived from $A$ and enriched with an abstract binary operation.
\end{theo}
\begin{proof}
Given the collection $Ob$ of prime filters in $A$, ordered by $R$, we assign to each lattice element $x$ the set consisting of all prime filters $P$ containing the element $x$, that is:
\[
h(x) = \{P \in Ob : x \in P\}
\]

We show that $h(x) \in {\cal O}_{R}$, that is $I_{\ms R} h(x) = h(x)$.

In fact, $(i)$ $I_Rh(x) \subseteq h(x)$. To prove the converse inclusion $(ii)$ $h(x) \subseteq I_{\ms R}h(x) = \bigcup \{R(P) : R(P) \subseteq h(x)\}$, let $P \in h(x)$, i.e.\ $x \in P$. Hence, for all $Q \in R(P)$, we have $P \subseteq Q$ and $x \in Q$, i.e.\ $P \in \bigcup \{R(P) : R(P) \subseteq h(x)\} = I_{\ms R}h(x)$.

\smallskip

Now recall that we have the following facts (\cite{Ras74}, p.58) :
\vspace{-0.2cm}
\begin{itemize}
\item[(h0)] $h$ is one-to-one, increasing, and $h(1) = Ob$
\item[(h1)] $h(a \wedge b) = h(a) \cap h(b)$ and $h(a \vee b) = h(a) \cup h(b)$
\item[(h2)] $h(\righthalfcap a) = \righthalfcap h(a)$
\end{itemize}
In order to complete the proof we need to show that: 
\begin{itemize}
\item[(h3)] $h(a \Rightarrow b) = h(a) \Rightarrow h(b)$
\item[(h4)] $h(a \dotdiv  b) = h(a) \dotdiv  h(b)$
\end{itemize}

$((h3) \rightarrow)$\ We remark that $h(a) \Rightarrow h(b) = I_{\ms R}(I_{\ms R}h(a) \supset I_{\ms R}h(b)) =\\
I_{\ms R}(-h(a) \cup h(b)) = \bigcup \{R(P) :  R(P) \subseteq -h(a) \cup h(b)\}$.

Let $P \in h(a \Rightarrow b)$, i.e.\ $a \Rightarrow b \in P$. We know that $P \in R(P)$. Let $Q \in R(P)$, i.e.\ $P \subseteq Q$. If $Q \not \in -h(a)$ then $Q \in h(a)$, i.e.\ $a \in Q$. Hence $a, a \Rightarrow b \in Q$. By modus ponens $b \in Q$, that is $Q \in h(b)$.

$((h3) \leftarrow)$\ Let $Q \in h(a) \Rightarrow h(b)$, then there is a prime filter $P_0$ such that $Q \in R(P_0)$ with $R(P_0) \subseteq -h(a) \cup h(b)$. 
If $b \in Q$ then $b \leq a \Rightarrow b \in Q$, i.e.\ $Q \in h(a \Rightarrow b)$. 
On account of $P_{0} \subseteq Q$, if $b \not \in Q$ then $a \not \in Q$. We consider the filter $F(Q,a)$ generated by $Q$ and $a$, that is $F(Q,a) = \{u :  a \Rightarrow u \in Q\}$. 
If $a \Rightarrow b \not \in Q$ we infer $b \not \in F(Q,a)$. By the well known Birkhoff-Stone theorem, there is a prime filter $Q'$ containing $a$ and $Q$ such that $b \not \in Q'$. This leads to $R(P_0) \not \subseteq -h(a) \cup h(b)$, a contradiction.

\medskip

$((h4) \rightarrow)$\ We remark that $h(a) \dotdiv h(b) = C_{\ms R}(I_{\ms R}h(a) \cap - I_{\ms R}h(b)) =\\
C_{\ms R}(h(a) \cap -h(b)) = \bigcup \{R(P) :  P \subseteq h(a) \cap -h(b)\}$

Let $P \in h(a \dotdiv b)$, i.e.\ $a \dotdiv b \leq a \in P$ (\cite{Raus74}, p.221),(\cite{War-Dil39}, p.337) then $P \in h(a)$. 
If $P \in -h(b)$ we have $P \in R(P)$ and $P \in h(a) \cap -h(b)$, that is $P \in h(a) \dotdiv h(b)$. 
If $P \not \in -h(b)$, then $b \in P$. We consider the ideal $I(-P,b)$ generated by the prime ideal $-P$ 
and $b$. We have $I(-P,b) = \{v : v \dotdiv b \in -P \}$.
Since $a \dotdiv b \not \in -P$ we infer $a \not \in I(-P,b)$. Again, by the Birkhoff-Stone result, we deduce that there is a prime ideal $I'$ containing $I(-P,b)$ and not containing $a$. Hence, $Q = -I' \subseteq P$ is a prime filter $Q\subseteq P$ such that $a \in Q$ and $b \not \in Q$. So, $P \in R(Q)$ with $Q \in h(a) \cap -h(b)$, that is $P \in h(a)\dotdiv  h(b)$.

$((h4) \leftarrow)$ Let $P \in h(a) \dotdiv h(b) = \bigcup \{R(Q) :  Q \subseteq h(a) \cap -h(b)\}$, then there is $Q_{0}$ such that $P \in R(Q_{0})$, $a \in Q_{0}$ and $b \not \in Q_{0}$. We have $a \leq a \vee b \in Q_{0}$ and $a \vee b = b \vee (a \dotdiv b) \in Q_{0}$ (see \cite{Raus74}, p.221). Since $Q_{0}$ is a prime filter we infer $(\alpha)$ $b \in Q_{0}$ or $(\beta)$ $a \dotdiv b \in Q_{0}$. Since $(\alpha)$ is impossible then $a \dotdiv b \in Q_{0} \subseteq P$, i.e.\ $P \in h(a \dotdiv b)$.
\end{proof}

Thus theorem \ref{repr-H-B} is proved.

\smallskip

\begin{coro}
For every three-valued {\L}ukasiewicz algebra $A$, there exists an isomorphism $h$ from $A$ into the $H$-$B$-algebra of sets ${\cal O}_{R}$ of the `concrete' Rauszer Boolean algebra derived from $A$ and enriched with an abstract binary operation.
\end{coro}
\begin{proof}
It is sufficient to remark that a three-valued {\L}ukasiewicz algebra is a $H$-$B$-algebra satisfying the following condition (\cite{Itu76b}, p.123):
\[
(T) \qquad (a \Rightarrow b) \vee (b \Rightarrow \righthalfcap \lefthalfcap a)= 1
\]

\noindent On ${\cal O}_{R}$, this equality becomes
\[
(T_{{\cal O}_{R}}) \quad (h(x) \Rightarrow h(y)) \cup (h(y) \Rightarrow h(\righthalfcap \lefthalfcap x)) = Ob
\]

\noindent Since 
\begin{align*}
h(\righthalfcap \lefthalfcap  x) &= h(\lefthalfcap x \Rightarrow 0) = I_{\ms R}(-h(\lefthalfcap x) \cup \emptyset) = I_{\ms R}(-h(1 \dotdiv x)) = \\
& = I_{\ms R}(-C_{\ms R}(h(1) \cap -h(x))) =
 I_{\ms R}(-C_{\ms R}-h(x)) = I_{\ms R}(I_{\ms S}h(x))
\end{align*}
we obtain that $(T_{{\cal O}_{R}})$ is equivalent to
\begin{align*}
&I_{\ms R}[-h(x) \cup h(y)] \cup I_{\ms R}[-h(y) \cup I_{\ms R}I_{\ms S}h(x)] = Ob
\end{align*}
\end{proof}

\noindent \textbf{B) Representation of symmetrical Heyting algebras}

\begin{remark}Some distributive lattices are equipped with a De Morgan negation or with a Kleene negation.
\end{remark}
We recall that a De Morgan algebra $(A, \wedge, \vee, \sim, 0, 1)$, or simply $A$, is an algebra of type $(2, 2, 1, 0, 0)$ such that $(A, \wedge, \vee, 0, 1)$ is a distributive lattice with zero $0$ and unit $1$, and $\sim$ fulfills the equalities:
\[
(DM_1) \sim\ \sim a = a \qquad \text{and} \qquad (DM_2) \sim (a \wedge b) =\ \sim a\ \vee \sim b.
\]

Moisil (\cite{Moisil35}, p.90) was the first to consider ``une  logique distributive dou\' ee d'une dualit\' e involutive $a \rightarrow \overline{a}$ ''. See also (\cite{Moisil72}, p.411).

\medskip

A De Morgan algebra $A$, in which the condition
\[
(K_{a,b})\quad a\ \wedge \sim a \leq b\ \vee \sim b, \quad \text{for any}\ a, b \in A
\]
\noindent holds, is called a \textbf{Kleene algebra}. J.A. Kalman \cite{Kalman58} has considered Kleene algebras under the name of `normal distributive i-lattices'.

\begin{defi}
A \textbf{symmetrical Heyting algebra} $(A, \wedge, \vee, \Rightarrow, \righthalfcap,  \sim, 1)$, or simply $A$, is an algebra of type $(2, 2, 2, 1, 1, 0)$, satisfying the following conditions (\cite{Mont80}, p.61):
\begin{itemize}
\item[] (SH1) $(A, \wedge, \vee, \Rightarrow, \righthalfcap, 1)$ is a Heyting algebra
\item[](SH2) $\sim\ \sim x = x$
\item[](SH3) $\sim (x \wedge y) =\ \sim x\ \vee \sim y$
\end{itemize}
\end{defi}

In these cases, the set $Ob$ of prime filters $P$ in the lattice $(A, \wedge, \vee, 0, 1)$, is enriched with an involution $\varphi$ of $Ob$. 

In fact, let $Ob$ be the set of all prime filters in $A$, and for every $P \in Ob$, let $\sim P =\{\sim p : p \in P\}$.

Let $\varphi : Ob \rightarrow Ob$ be the Bia{\l}ynicki-Birula and Rasiowa (\cite{Ras74}, pp.45-46) mapping defined by:
\[
\varphi(P) = -(\sim P)
\]

The set $\varphi(P)$ is a prime filter. $\varphi$ is a one-to-one mapping from $Ob$ onto $Ob$ such that, for all $P \in Ob$: 
\[
\varphi(\varphi (P)) = P
\]

This mapping determines a De Morgan operation $\sim$ on ${\cal P}(Ob)$ in the following way: 
\[
\sim X = - \varphi (X), \quad \text{for any}\quad X \subseteq Ob.
\]

If A is a Kleene algebra, then the involution $\varphi$ fulfills --in addition-- the following condition (Bia{\l}ynicki-Birula and Rasiowa, 1958): 
\[
(K) \qquad P \subseteq \varphi(P) \quad \text{or} \quad \varphi(P) \subseteq P.
\]

If $(A, \wedge, \vee, \Rightarrow, \righthalfcap,  \sim, 1)$ is a symmetrical Heyting algebra (resp.\ a $H-B$-algebra), then the set ${\cal O}_{R}$ has some special properties.

\medskip

If $G \in {\cal O}_{R}$ then $\sim G \in {\cal O}_{R}$.

Assume $G$ is $R$-open. For one side we have $(i)$ $I_{\ms R}(\sim G) \subseteq\ \sim G$. To prove the converse inclusion, $(ii)$ $\sim G \subseteq I_{\ms R}(\sim G)$ let $P \in\ \sim G = - \varphi(G)$. That is $P \not \in \varphi(G)$ and $(a)$ $\varphi(P) \not \in G$.
If $P \not \in I_{\ms R}(\sim G) = \bigcup \{R(Q) : R(Q) \subseteq - \varphi(G)\}$, then $R(P) \not \subseteq - \varphi(G)$. In this case there is $Q$ such that $P \subseteq Q$ and $Q \in \varphi(G)$. Thus $\varphi(Q) \subseteq \varphi(P)$ and $\varphi(Q) \in G \in{\cal O}_{R}$. Hence $R(\varphi(Q)) \subseteq G$ and \\
$\varphi(P) \in R(\varphi(Q)) \subseteq G$, which contradicts $(a)$.

\smallskip

Therefore the system $({\cal O}_{R}, \cap, \cup, \Rightarrow, \dotdiv, \sim, \emptyset, Ob)$ is a \textbf{symmetrical $\textbf{H}$-$\textbf{B}$-algebra}.

\

If $(A, \wedge, \vee, \Rightarrow, \righthalfcap,  \sim, 1)$ is a symmetrical Heyting algebra, then the operation $\dotdiv$ can be expressed on ${\cal O}_{R}$ in terms of $\sim$ and $\Rightarrow$, in the following way:
\[
G \dotdiv H =\ \sim ( \sim H \Rightarrow\ \sim G)
\]

Indeed, we will prove that:
\begin{itemize}
\item[]$(i)$ $G \subseteq [H\ \cup\ \sim (\sim H \Rightarrow\ \sim G)]$
\item[]$(ii)$ $\text{if}\ G \subseteq H \cup X$, \text{then} $\sim ( \sim H \Rightarrow\ \sim G) \subseteq X$
\end{itemize}

The proof of $(i)$ follows from the following equivalences on account of the intuitionistic equality $x \wedge (x \Rightarrow y) = x \wedge y$ (\cite{Ras74}, p.55):
\begin{align*}
G &\subseteq [H\ \cup\ \sim (\sim H \Rightarrow\ \sim G)] \Longleftrightarrow\ \sim [H\ \cup \sim (\sim H \Rightarrow\ \sim G)]\subseteq\ \sim G\\ 
&\Longleftrightarrow\ \sim H \cap (\sim H \Rightarrow\ \sim G) =\ \sim H\ \cap \sim G \subseteq\ \sim G. 
\end{align*}

To prove $(ii)$, assume $G \subseteq H \cup X$. Hence $\sim H\ \cap \sim X \subseteq\ \sim G$. So,\\ $\sim X \subseteq\ \sim H \Rightarrow\ \sim G$, i.e. $\sim (\sim H \Rightarrow\ \sim G) \subseteq X$.

The operation $E(G, H) =\ \sim (\sim H \Rightarrow\ \sim G)$ was introduced by Moisil (\cite{Moisil72}, p.412) to represent the \textbf{`exception'}.

\begin{theo} \label{repr-SH}
For every symmetrical Heyting algebra $A$, there exists an isomorphism $h$ from $A$ into the symmetrical $H$-$B$-algebra of sets ${\cal O}_{R}$ of the `concrete' Rauszer algebra derived from $A$ and enriched with an abstract binary operation.
\end{theo}
\begin{proof}
By results above we remark that $h(x) \in {\cal O}_{R}$ and $\sim h(x) \in {\cal O}_{R}$.

The proof of $h(\sim x) =\ \sim h(x)$ (see \cite{Ras74}, p.46) follows from the following equivalences:
\begin{align*}
P \in h(\sim x) &\Leftrightarrow\ \sim x \in P \Leftrightarrow x \in\ \sim P \Leftrightarrow x \not \in -(\sim P) = \varphi(P) \\&\Leftrightarrow 
\varphi(P) \not \in h(x) \Leftrightarrow P \not \in \varphi(h(x)) \Leftrightarrow P \in - \varphi(h(x)) =\ \sim h(x).
\end{align*}

Some steps in the proof of Theorem \ref{repr-H-B} complete the statement.
\end{proof}

\

We close this section with the following facts, which lead us to envisage the representation theorem in the next case.

On $({\cal O}_{R}, \cap, \cup, -, \Rightarrow, \sim, \emptyset, Ob)$, we can consider the binary implication $\rightarrow_{w}$ defined by (\cite{Mont63a}, p.361):
\[
G \rightarrow_{w}H = G \Rightarrow (\sim G \cup H),\quad \text{for} \ G, H \in {\cal O}_{R}
\]

We remark that: 
\begin{align*}
G \rightarrow_{w} H = Ob &\quad  \text{iff } \quad  G \Rightarrow (\sim G \cup H) = Ob\\
&\quad \text{iff} \quad G\ \subseteq\ (\sim G \cup H) \quad \text{iff} \quad\ G = G\ \cap (\sim G \cup H).
\end{align*}

\begin{propo}\label{arrow-W}
Using properties of Kleene symmetrical Heyting algebras, it follows that, for all $G, H, K \in {\cal O}_{R}$ we have:
\begin{itemize}
\item[] (M0) $G\ \cap \sim G \subseteq H\ \cup \sim H$
\item[](M1) $G \rightarrow_{w}G = Ob$
\item[](M2) $G \cap (G \rightarrow_{w}H) = G \cap (\sim G \cup H)$
\item[](M3) $G \rightarrow_{w}(H \rightarrow_{w}K) \subseteq (G \cap H)\rightarrow_{w}K$
\end{itemize}
\end{propo}
\begin{proof}
To prove $(M0)$ assume $P \in G\ \cap \sim G$. That is $P \in G$ and $\varphi(P) \not \in G$. So, $\varphi(P) \not \in R(P)$ and $(i)$ $P \not \subseteq \varphi(P)$. If $P \not \in H\ \cup \sim H$ then $P \not \in H$ and $\varphi(P) \in H$. Hence $(ii)$ $\varphi(P) \not \subseteq P$. $(i)$ and $(ii)$ imply that the condition $(K)$ is not satisfied, a contradiction. Thus $({\cal O}_{R}, \cap, \cup, -, \Rightarrow, \sim, \emptyset, Ob)$ is a Kleene symmetrical Heyting algebra.

Since $G \subseteq ( \sim G\ \cup\ G)$ we deduce that $(M1)$ $G \rightarrow_{w}G =
G \Rightarrow (\sim\ G \cup\ G) = Ob$. 

Also, $(M2)$ $G \cap (G \rightarrow_{w}H) = G \cap (G \Rightarrow ( \sim G \cup H)) = G \cap (\sim G \cup H)$ (on the account of the intuitionistic equality 
$x \wedge (x \Rightarrow y) = x \wedge y$ (\cite{Ras74}, p.55)).

Next, we prove the inclusion:

\noindent $(M3)$ $G \Rightarrow (\sim G \cup (H \Rightarrow (\sim H \cup K))) \subseteq (G \cap H) \Rightarrow (\sim G\ \cup \sim H \cup K)$, \\
which in a Heyting algebra is equivalent to \\
$(G \cap H) \cap [G \Rightarrow (\sim G \cup (H \Rightarrow (\sim H \cup K)))] \subseteq (\sim G\ \cup \sim H \cup K)$.

\smallskip

On twice account of the equality $x \wedge (x \Rightarrow y) = x \wedge y$ we obtain:
\begin{flalign*}
H\ \cap\ & G \cap [G \Rightarrow (\sim G \cup (H \Rightarrow (\sim H \cup K)))] =\\
&H \cap [G \cap (\sim G \cup (H \Rightarrow (\sim H \cup K)))] =\\
&H \cap [(G\ \cap \sim G) \cup (G \cap (H \Rightarrow (\sim H \cup K)))] =\\
&(H \cap G\ \cap \sim G) \cup (G \cap H \cap (H \Rightarrow (\sim H \cup K))) =\\
&(H \cap G\ \cap \sim G)\ \cup (G \cap H \cap (\sim H \cup K)) =\\
&(G \cap H) \cap (\sim G\ \cup \sim H \cup K) \subseteq\ \sim G\ \cup \sim H \cup K
\end{flalign*}
\end{proof}
This completes the proof of Proposition \ref{arrow-W}

\newpage

\noindent \textbf{C) Representation of Nelson algebras}

\begin{defi}\label{Nelson}
A \textbf{Nelson algebra} $(A, \wedge, \vee, \rightarrow_{w}, \sim, 1)$, or simply $A$, is an algebra of type $(2, 2, 2, 1, 0)$, satisfying the following conditions (\cite{Mont-Mont73}, p.3):
\begin{itemize}
\item[](N0) $(A, \wedge, \vee,  \sim, 1)$ is a \textbf{Kleene algebra}
\item[](N1) $a \rightarrow_{w}a = 1$
\item[](N2) $a \wedge (a \rightarrow_{w}b) = a \wedge (\sim a \vee b)$
\item[](N3) $a \rightarrow_{w}(b \rightarrow_{w}c) = (a \wedge b) \rightarrow_{w}c$
\end{itemize}
\end{defi}

If $A$ is a Nelson algebra, $\rightarrow_{w}$ is called the \textbf{weak implication} sign and $\sim$ is the \textbf{strong negation}. Let us put $\sim 1 = 0$. Also a \textbf{weak negation} $\righthalfcup_{w}$ can be defined in the following way: $\righthalfcup_{w}\, a = a \rightarrow_{w}0$

From $(N1)$ and $(N2)$ we get that $1$ is the last element of $A$ \cite{Mont-Mont73}. For a detailed study of this structure, see \cite{Ras58}, \cite{Bri-Mont67}, \cite{Mont80}, \cite{Ras74}.

For Nelson algebras, the kernel of a homomorphism from a Nelson algebra into another, is a deductive system in regard to $\rightarrow_{w}$. 

In particular, deductive systems are filters (see (\cite{Mont63b}, p.4), (\cite{Ras74}, p.91)).

\medskip

It is know (see for example (\cite{Ras74}, p.70)  that the Nelson implication satisfies the equalities $(I1)-(I3)$ of Definition \ref{ded-alg}. That is, the system $(A, \rightarrow_{w}, Ob)$ is a deductive algebra. Therefore, the deduction theorem is satisfied and the deductive system generated by a set $X \not \equal \emptyset$ is (\cite{Mont63b}, p.5):
\[
D(X) = \{z \in A : (x_{1} \wedge x_{2} \wedge \ldots \wedge x_{n}) \rightarrow_{w}z = 1, \text{with}\ x_{1},  x_{2}, \ldots , x_{n} \in X\}
\]

In addition, we note that Rasiowa has shown that $(a\ \wedge \sim a) \rightarrow_{w} b = 1$ for any $a, b \in A$ and also $(a \wedge \righthalfcup_{w} a) \rightarrow_{w} b = 1$ for any $a, b \in A$ (\cite{Ras58}, p.65, \cite{Ras74}, p.68). This implies that if a deductive system $D$ is proper, the fact $a, \sim a \in D$, for any $a \in A$ leads to a contradiction.

\

Let $Ob$ be the set of all {\bf{prime filters}} in $A$, ordered by $\subseteq$. In (\cite{Ras74}, (4.21, 4.17), p.98); (\cite{Ras58}, (3.9), p.79) it is proved that $Ob$ is the union of two subsets $Ob = Ob_1\ \cup\ Ob_2$ such that $Ob_1$ is the set of all prime deductive systems and $\varphi(Ob_1) = Ob_2$.

In addition, if $P \in Ob_1$, then $P \subseteq \varphi(P)$, whereas if $P \in Ob_2$, then $\varphi(P) \subseteq P$ (\cite{Ras74},(4.19, 4.20), p.97).

According to the terminology in \cite{Ras74}, \cite{Ras58}, elements in $Ob_1$ are called prime s.f.f.k.\ (special filter of the first kind) and those in $Ob_2$ are named prime s.f.s.k.\ (special filter of the second kind).

\

We recall the following keystone result.
\begin{propo}\label{int-pro}

In a Nelson algebra $A$, the involution $\varphi$ has the \textbf{interpolation property} (\cite{Mont63a},p.361), that is: 

\noindent $(\varphi_{\ms {inter}}):$ \qquad for all $P, Q \in Ob$ satisfying the conditions:

$(\alpha)$ $P \subseteq \varphi (P)$, \quad $(\beta)$ $Q \subseteq \varphi (Q)$, \quad $(\gamma)$ $P \subseteq \varphi (Q)$, \quad $\{\, (\delta)$ $Q \subseteq \varphi (P)\, \}$

\noindent there is a prime filter $M$ such that:

\qquad \qquad $P \subseteq M$, \quad $Q \subseteq M$, \quad $M \subseteq \varphi (P)$, \quad $M \subseteq \varphi (Q)$.
\end{propo}
\begin{proof}
In fact, let $P, Q \in Ob$ satisfying the conditions $(\alpha)-(\delta)$ and $X = P \cup Q$. We consider the least deductive system $D = D(X)$ containing $P$ and $Q$. Since $(A, \rightarrow_{w},1)$ is a deductive algebra we have: 

\noindent $D(X) = \{z \in A : (p \wedge q) \rightarrow_{w} z = 1, \text{with}\ p \in P, q \in Q\}$. In particular $p \wedge q \in D(X)$, for all $p \in P, q \in Q$. The deductive system $D(X)$ is proper; in fact, if $0 \in D(X)$ then $(p \wedge q) \rightarrow_{w} 0 = \righthalfcup_{w}(p \wedge q) = 1 \in \varphi(P)$. Hence $(p \wedge q) \wedge \righthalfcup_{w}(p \wedge q) \rightarrow_{w} t =\\
 1 \in \varphi(P)$, for all $t \in A$, a contradiction.
Since $D(X)$ is a deductive system and $0 \not \in D(X)$, then there is an irreducible (= prime) deductive system $M$ such that $D(X) \subseteq M$ (\cite{Ras74}, 4.16., p.96).

If $M \not \subseteq \varphi(P)$ then there is $u \in M$ such that $u \not \in \varphi(P) = - (\sim P)$. So, $\sim u \in P \subseteq M$. Consequently $u\ \wedge \sim u \in M$. Since $(u\ \wedge \sim u) \rightarrow_{w} t = 1 \in M$, for all $t \in A$, we get a contradiction.
By a similar argument we show $M \subseteq \varphi(Q)$. 

Note that in the statement, the prime filter $M$ can be replaced by $\varphi(M)$.

This completes the proof of Proposition \ref{int-pro}.
\end{proof}

\

The following crucial characterization has been announced by A. Monteiro (\cite{Mont63a}, p.361).
Let $(A, \wedge, \vee,  \sim, 1)$ be a Kleene algebra such that for each pair $(a, b)$ of elements there is the intuitionistic implication $a \Rightarrow (\sim a \vee b)$. If  we define $a \rightarrow_{w}b = a \Rightarrow (\sim a \vee b)$ then the following conditions are equivalent:
\begin{itemize}
\item[]$(N)$ $(a \wedge b) \rightarrow_{w}c \leq a \rightarrow_{w}(b \rightarrow_{w}c)$
\item[]$(\varphi_{\ms {inter}})$ The Bia{\l}ynicki-Birula and Rasiowa mapping $\varphi$ on $(Ob, \subseteq)$ satisfies the interpolation property. 
\end{itemize}

In view of the result above we need only to prove $(\varphi_{\ms {inter}}) \rightarrow (N)$. 

Using the definition of $\rightarrow_{w}$ the inequality above\\ 
$(a \wedge b) \rightarrow_{w}c \leq a \rightarrow_{w}(b \rightarrow_{w}c) $ can be rewritten as
\begin{center}
$[(a \wedge b) \Rightarrow (\sim a\ \vee \sim b \vee c)] \leq a \Rightarrow [\sim a \vee (b \Rightarrow (\sim b \vee c))]$ 
\end{center}
\noindent By the definition  of the intuitionistic implication  this is equivalent to
\begin{center}
$a \wedge [(a \wedge b) \Rightarrow (\sim a\ \vee \sim b \vee c)] \leq\ \sim a \vee [b \Rightarrow (\sim b \vee c)]$
\end{center}
\noindent and on account of the intuitionistic equalities $(x \wedge y) \Rightarrow z = x \Rightarrow (y \Rightarrow z)$ and $x \wedge (x \Rightarrow y) = x \wedge y$ the last inequality is equivalent to
\begin{center}
$a \wedge [b \Rightarrow (\sim a\ \vee \sim b \vee c)] \leq\ \sim a \vee [b \Rightarrow (\sim b \vee c)]$
\end{center}
If $a \wedge [b \Rightarrow (\sim a\ \vee \sim b \vee c)] \not \leq\ \sim a \vee [b \Rightarrow (\sim b \vee c)]$ then there is a prime filter $P \in Ob$ such that

$\ms{\textbf{(1)}}$ $a \wedge [b \Rightarrow (\sim a\ \vee \sim b \vee c)] \in P$ while

$\ms{\textbf{(2)}}$ $\sim a \vee [b \Rightarrow (\sim b \vee c)] \not \in P$

\smallskip

\noindent From $\ms{\textbf{(1)}}$ we get $\ms{\textbf{(3)}}$ $a \in P$ and $\ms{\textbf{(4)}}$ $b \Rightarrow (\sim a\ \vee \sim b \vee c) \in P$. 

\noindent By $\ms{\textbf{(2)}}$ we infer $\ms{\textbf{(5)}}$ $\sim a \not \in P$ and $\ms{\textbf{(6)}}$ $b \Rightarrow (\sim b \vee c) \not \in P$. From $\ms{\textbf{(6)}}$, and the intuitionistic inequality  $y \leq x \Rightarrow y$  we deduce $\ms{\textbf{(7)}}$ $(\sim b \vee c) \not \in P$ and in turn we get $\ms{\textbf{(8)}}$ $\sim b \not \in P$ and $\ms{\textbf{(9)}}$ $c \not \in P$. 

\noindent According to $\ms{\textbf{(5)}}$, $\ms{\textbf{(8)}}$ and $\ms{\textbf{(9)}}$ we obtain $\ms{\textbf{(10)}}$ $(\sim a\ \vee \sim b \vee c) \not \in P$.

\noindent From $\ms{\textbf{(4)}}$ and $\ms{\textbf{(10)}}$ by modus ponens we infer $\ms{\textbf{(11)}}$ $b \not \in P$.

\smallskip

Let $D(P, b) = \{x \in A : b \rightarrow_{w}x \in P\}$ be the deductive system generated by $P$ and $b$. It is proper since $c \not \in D(P, b)$. In fact, by $\ms{\textbf{(6)}}$, $b \rightarrow_{w}c =\\
 b \Rightarrow (\sim b \vee c) \not \in P$.

\smallskip

Since $A$ is a Kleene algebra, two cases are to be considered: $P \subseteq \varphi(P)$ or $\varphi(P) \subseteq P$.

If $\varphi(P) \subseteq P$ then by $\ms{\textbf{(11)}}$ we get $\sim b \in \varphi(P)$ and thus \\
$(\sim a\ \vee \sim b \vee c) \in \varphi(P) \subseteq P$, which contradicts $\ms{\textbf{(10)}}$. Hence $(\alpha)$ $P \subseteq \varphi(P)$.

\smallskip

Since $D(P, b)$ is a proper deductive system and $c \not \in D(P, b)$, then by (\cite{Ras74}, 4.16., p.96) there is an irreducible (= prime)  deductive system $P_{b}$ such that $D(P, b) \subseteq P_{\ms{b}}$ and $c \not \in P_{\ms{b}}$. Thus $(b \rightarrow_{w}c) \not \in P_{\ms{b}}$. Therefore $(\beta)$ $P \subseteq P_{\ms{b}}$ and\\ $\ms{\textbf{(12)}}$ $c \not \in P_{\ms{b}}$.
From $\ms{\textbf{(3)}}$ and $(\beta)$ we obtain $a \in P \subseteq P_{\ms{b}}$ and by construction $b \in P_{\ms{b}}$. Hence $\ms{\textbf{(13)}}$ $(\sim a\ \vee \sim b) \not \in P_{\ms{b}}$ because $P_{\ms{b}}$ is proper. 
In addition, since $b \in P_{\ms{b}}$ we get $\ms{\textbf{(14)}}$ $\sim b \not \in P_{\ms{b}}$, that is $\ms{\textbf{(15)}}$ $b \in \varphi(P_{\ms{b}})$. 

\smallskip

Concerning the prime filter $P_{\ms{b}}$ two cases are to be considered: $P_{\ms{b}} \subseteq \varphi(P_{\ms{b}})$ or $\varphi(P_b) \subseteq P_{\ms{b}}$.

Assume $P_{\ms{b}} \subseteq  \varphi(P_{\ms{b}})$. From $\ms{\textbf{(4)}}$ we have $b \Rightarrow (\sim a\ \vee \sim b \vee c) \in P \subseteq P_{\ms{b}} \subseteq \varphi(P_{\ms{b}})$ and by $\ms{\textbf{(15)}}$ we infer $(\sim a\ \vee \sim b \vee c) \in P \subseteq P_{\ms{b}} \subseteq \varphi(P_{\ms{b}})$.
By $\ms{\textbf{(3)}}$ we infer $a \in P \subseteq P_{\ms{b}} \subseteq \varphi(P_{\ms{b}})$ so $\sim a \not \in \varphi(P_{\ms{b}})$. 
From $\ms{\textbf{(15)}}$ we get $\sim b \not \in \varphi(P_{\ms{b}})$. Thus $c \in \varphi(P_{\ms{b}})$. Since $P \subseteq P_{\ms{b}} \subseteq \varphi(P_{\ms{b}})$ we have $\varphi(P_{\ms{b}}) \subseteq \varphi(P)$. It means that $c \in \varphi(P)$. In other words, $\sim c \not \in \varphi(P)$, $\sim c \in\ \sim P$ and $c \in P$, which contradicts $\ms{\textbf{(9)}}$. Hence $(\gamma)$ $\varphi(P_{\ms{b}}) \subseteq P_{\ms{b}}$.

\smallskip
 
\noindent In short,

$(\alpha)$ $P \subseteq \varphi(P)$, \quad  $(\gamma)$ $\varphi(P_{\ms{b}}) \subseteq P_{\ms{b}}$, \quad $(\beta)$ $P \subseteq P_{\ms{b}}$, \quad $\{\, (\delta)$ $\varphi(P_{\ms{b}}) \subseteq \varphi(P)\, \}$
 
\noindent By the interpolation property $(\varphi_{\ms {inter}})$ there is a prime filter $M$ such that:
 
\qquad \qquad $P \subseteq M$, \quad $\varphi(P_{\ms{b}}) \subseteq M$, \quad $M \subseteq \varphi(P)$, \quad and \quad $M \subseteq P_{\ms{b}}$.

\smallskip

From $\ms{\textbf{(4)}}$ we have $b \Rightarrow (\sim a\ \vee \sim b \vee c) \in P \subseteq M$. By $\ms{\textbf{(15)}}$ and $\varphi(P_{\ms{b}}) \subseteq M$ we get $(\sim a\ \vee \sim b \vee c) \in M$.

From $\ms{\textbf{(3)}}$ and $(\alpha)$ we deduce $\sim a \not \in \varphi(P)$. This fact and $\ms{\textbf{(14)}}$ entail $\sim a \not \in M$ and $\sim b \not \in M$. Hence $c \in M \subseteq P_{\ms{b}}$ which contradicts $\ms{\textbf{(12)}}$.

This completes the proof of the statement $(\varphi_{\ms {inter}}) \rightarrow (N)$.

\begin{propo}\label {HB-Nelson}
If the involution $\varphi$ of $Ob$ satisfies the interpolation property then, the Kleene symmetrical Heyting algebra of sets $({\cal O}_{R}, \cap, \cup, \rightarrow_{w}, \sim, Ob)$ where, for $G, H \in {\cal O}_{R}$, $G\rightarrow_{w} H = G \Rightarrow (\sim G \cup H)$, is a Nelson algebra. 
\end{propo}
\begin{proof}
On account of Proposition \ref{arrow-W} we need to show the following inequality 
\[(G \cap H) \rightarrow_{w} K \subseteq G \rightarrow_{w} (H \rightarrow_{w} K)
\]
Using the definition of $\rightarrow_{w}$ this inequality can be rewritten as
\[
(G \cap H) \Rightarrow (\sim G\ \cup \sim H \cup K) \subseteq G \Rightarrow (\sim G \cup (H \Rightarrow (\sim H \cup K)))
\]
By the definition of the intuitionistic implication this is equivalent to 
\[
G \cap [(G \cap H) \Rightarrow (\sim G\ \cup \sim H \cup K)] \subseteq\ \sim G \cup (H \Rightarrow (\sim H \cup K))
\]
and on account of the intuitionistic equalities $(x \wedge y) \Rightarrow  z = x \Rightarrow  (y \Rightarrow  z)$ and $x \wedge (x \Rightarrow  y) = x \wedge y$ the preceding inequality is equivalent to
\[
G \cap (H \Rightarrow (\sim G\ \cup \sim H \cup K)) \subseteq\ \sim G \cup (H \Rightarrow  (\sim H \cup K))
\]

In other words, we need to show that the set below reported by $(A)$ is included in that indicated by $(B)$:
\vspace{-0.2cm}
\begin{flalign*}
(A) \quad &G \cap (H \Rightarrow (\sim G\ \cup \sim H \cup K)) =
G \cap I_{\ms R}(I_{\ms R}H \supset I_{\ms R}(\sim G\ \cup \sim H \cup K)) =\\
&G \cap I_{\ms R}(- H\ \cup \sim G\ \cup \sim H \cup K) =\\
&G \cap \bigcup \{R(Q) : R(Q) \subseteq - H\ \cup \sim G\ \cup \sim H \cup K\}&
\end{flalign*}
\begin{flalign*}
(B) \quad &\sim G \cup (H \Rightarrow  (\sim H \cup K)) =\ \sim G \cup I_{\ms R}(I_{\ms R}H \supset I_{\ms R}(\sim H \cup K)) =\\
&\sim G \cup I_{\ms R}(- H\ \cup \sim H \cup K)) =\ \sim G \cup \bigcup \{R(Q) : R(Q) \subseteq - H\ \cup \sim H \cup K\}&
\end{flalign*}

\smallskip

Assume there is a prime filter $P \in Ob$ such that $P \in (A)$. That is $\ms{\textbf{(1)}}$ $P \in G$ and $\ms{\textbf{(2)}}$ $P \in \bigcup \{R(Q) : R(Q) \subseteq - H\ \cup \sim G\ \cup \sim H \cup K\}$. 

\smallskip

If $\varphi(P) \not \in G$ then $P \not \in \varphi(G)$; that is $P \in -\varphi(G) =\ \sim G \in (B)$. The result is obtained. 

We consider now the case: $\varphi(P) \in G$. Thus $P$ and $\varphi(P)$ are in $G$. 
We remark that, in this particular situation, the place of $P$ and $\varphi(P)$ are interchangeable.
We can suppose for example that $(\alpha)$ $P \subseteq \varphi(P)$. 

By $\ms{\textbf{(2)}}$ let $U \in Ob$ be a prime filter such that $P \in R(U)$. If
 
\noindent $R(U) \subseteq -H\ \cup \sim H \cup K$, then the result follows.

Otherwise $R(U)\ \cap \sim G \not \equal \emptyset$. Let $V \in Ob$ be a prime filter such that $V \in R(U)$ and $V \in\ \sim G = -\varphi(G)$ so $V \not \in \varphi(G)$, i.e.\ $\ms{\textbf{(3)}}$ $\varphi(V) \not \in G$. 

\smallskip

We note that $U \subseteq P$, $U \subseteq V$, $\varphi(P) \in G$ and $V \not \subseteq P$ (if $V \subseteq P$ then \\ $\varphi(P) \subseteq \varphi(V) \in G$ a contradiction). Also, $P \subseteq \varphi(P) \subseteq \varphi(U)$ and $\varphi(V) \subseteq \varphi(U)$.

Two cases are to be considered: $V \subseteq \varphi(V)$ or $\varphi(V) \subseteq V$.

\

\noindent Case I. $V \subseteq \varphi(V)$. Let $D = D(P \cup \varphi(V))$ be the deductive system generated by $P \cup \varphi(V)$. We have that $D$ is proper because $D \subseteq \varphi(U)$.

From (\cite{Ras74}, 4.16., p.96) there is an irreducible (= prime) deductive system $P_{v}$ such that $(\gamma)$ $P \subseteq (P \cup \varphi(V)) \subseteq D \subseteq P_{\ms{v}}$. Since $\varphi(V) \subseteq P_{\ms{v}}$ we infer\\ 
$\varphi(P_{\ms{v}}) \subseteq V \subseteq \varphi(V) \subseteq P_{\ms{v}}$, i.e. $(\beta)$ 
$\varphi(P_{\ms{v}}) \subseteq P_{\ms{v}}$.

\noindent In summary,

$(\alpha)$ $P \subseteq \varphi(P)$, \quad $(\beta)$ $\varphi(P_{\ms{v}}) \subseteq P_{\ms{v}}$, \quad $(\gamma)$ $ P \subseteq P_{\ms{v}}$, \quad $\{\, (\delta)$ $\varphi(P_{\ms{v}}) \subseteq \varphi(P)\, \}$.

\noindent By the interpolation property $(\varphi_{\ms {inter}})$ there is a prime filter $W$ (id.\ $\varphi(W))$ such that:

\qquad \qquad $P \subseteq W$, \quad $\varphi(P_{\ms{v}}) \subseteq W$, \quad $W  \subseteq \varphi(P)$, \quad and \quad $W \subseteq P_{\ms{v}}$.

\medskip

From here, we deduce that $\ms{\textbf{(4)}}$ $\varphi(V) \subseteq W$. Indeed, let $v \in \varphi(V) \subseteq D \subseteq P_{\ms{v}}$, so $\sim v \not \in P_{\ms{v}}$ since $P_{\ms{v}}$ is proper; thus $v \not \in\ \sim P_{\ms{v}}$, i.e.\ $v \in \varphi(P_{\ms{v}}) \subseteq W$.

In addition we note that $W \in G$, as $P \subseteq W$ and $P \in G$. Also, $W \subseteq \varphi(P)$ implies $P \subseteq \varphi(W)$ and $\varphi(W) \in G$.

From $\ms{\textbf{(3)}}$ we deduce that $\varphi(P_{\ms{v}}) \not \in G$. In fact, since $\varphi(V) \subseteq P_{\ms{v}}$ then\\ $\varphi(P_{\ms{v}}) \subseteq V \subseteq \varphi(V) \not \in G$.

\medskip

Concerning the prime filter $W$, two cases are to be considered: $W \subseteq \varphi(W)$ or $\varphi(W) \subseteq W$.

Assume $W \subseteq \varphi(W)$. Since by $\ms{\textbf{(4)}}$ $V \subseteq \varphi(V) \subseteq W$, we get $\varphi(W) \subseteq  \varphi(V)$ so $\varphi(V) \in G$, in contradiction with $\ms{\textbf{(3)}}$.

We remark that, in this particular situation, the places of $W$ and $\varphi(W)$ are interchangeable, so the case $\varphi(W) \subseteq W$ leads also to a contradiction.

\

\noindent Case II. $\varphi(V) \subseteq V$. Let $D = D(P \cup V)$ be the deductive system generated by $P \cup V$, that is $D = D(P \cup V) = \{z \in A : (p \wedge v) \rightarrow_{w}  z = 1, \text{with}\ p \in P, v \in V \}$. We prove that $D$ is proper. In fact if $0 \in D$ there are $p \in P, v \in V$ such that $(p \wedge v) \rightarrow_{w} 0 = 1$. Hence $(p \rightarrow_{w}(v \rightarrow_{w} 0)) = (p \rightarrow_{w} \righthalfcup_{w} v) = 1 \in P$. Since $p \in P$ we deduce $\righthalfcup_{w}\, v = v \rightarrow_{w} 0 =
v \Rightarrow (\sim v \vee 0) = (v \Rightarrow\ \sim v) \in P$. If $v \in P$ then $\sim v \in P$, a contradiction. Hence $v \not \in P$, for all $v \in V$, that is $P \cap V = \emptyset$, a contradiction because $1 \in P \cap V$.

\medskip

From (\cite{Ras74}, 4.16., p.96) there is an irreducible (= prime) deductive system $P_{v}$ such that $(\gamma)$ $P \subseteq (P \cup V) \subseteq D \subseteq P_{\ms{v}}$. Since $\varphi(V) \subseteq V\subseteq P_{\ms{v}}$ we infer\\
 $\varphi(P_{\ms{v}}) \subseteq V \subseteq P_{\ms{v}}$, i.e. $(\beta)$ $\varphi(P_{\ms{v}}) \subseteq P_{\ms{v}}$.

\noindent In summary,

$(\alpha)$ $P \subseteq \varphi(P)$, \quad $(\beta)$ $\varphi(P_{\ms{v}}) \subseteq P_{\ms{v}}$, \quad $(\gamma)$ $ P \subseteq P_{\ms{v}}$, \quad $\{\, (\delta)$ $\varphi(P_{\ms{v}}) \subseteq \varphi(P)\, \}$.

\noindent By the interpolation property $(\varphi_{\ms {inter}})$ there is a prime filter $W$ (id.\ $\varphi(W))$ such that:

\qquad \qquad $P \subseteq W$, \quad $\varphi(P_{\ms{v}}) \subseteq W$, \quad $W  \subseteq \varphi(P)$, \quad and \quad $W \subseteq P_{\ms{v}}$.

\medskip

\noindent The rest of the proof is similar to the previous one.

This completes the proof of Proposition \ref{HB-Nelson} 
\end{proof}

\

To achieve our project, we will prove a representation theorem for Nelson algebras. 
For other representation theorems concerning this type of algebras, see \cite{Ras74} and \cite{Vak77}.

\begin{theo}\label{repr-Nelson}
For every Nelson algebra $A$, there exists an isomorphism $h$ from $A$ into the Nelson algebra of sets ${\cal O}_{R}$ of the `concrete' Rauszer algebra derived from $A$ and enriched with an abstract binary operation.
\end{theo}
\begin{proof} 
Since $(A, \wedge, \vee, \sim, 1)$ is a Kleene algebra, we infer that the conditions $h(0)-h(3)$ in the proof of Theorem \ref{repr-H-B} and the equality $h(\sim a) =\ \sim h(a)$ in Theorem \ref{repr-SH} are satisfied. 
Also $h(K_{a,b}) = K_{h(a),h(b)}$.

\medskip

Now we will prove that $h(a \rightarrow_{w}b) = h(a) \rightarrow_{w}h(b)$, where $h(a) \rightarrow_{w}h(b) = h(a) \Rightarrow (\sim h(a) \cup h(b))$.

\medskip

\noindent First we remark that 

$h(a) \rightarrow_{w}h(b) = h(a) \Rightarrow (\sim h(a) \cup h(b)) = I_{\ms R}(h(a) \supset (\sim h(a) \cup h(b))) = \\
I_{\ms R}(-h(a)\ \cup \sim h(a) \cup h(b)) = \bigcup \{R(P) : R(P) \subseteq -h(a) \cup h(\sim a) \cup h(b)\}$.

\

$(\rightarrow)$\ Let $P \in h(a \rightarrow_{w}b), i.e.\ a \rightarrow_{w}b \in P$. 

We know that $P \in R(P)$. Let $Q \in R(P)$, i.e.\ $P\subseteq Q \in Ob$. We will prove that $Q \in -h(a) \cup h(\sim a) \cup h(b)$. 

If $Q \not \in -h(a)$, i.e.\ $a \in Q$, we have $a, a \rightarrow_{w}b \in Q$. Since $Q$ is a filter we infer, by $(f_2)$ and $(N2)$ in Definition \ref{Nelson}  that $a \wedge (a \rightarrow_{w}b) = a \wedge (\sim a \vee b) \in Q$. Hence, by $(f_3)$, $\sim a \vee b \in Q$, that is $Q \in h( \sim a \cup b) =\ h(\sim a) \cup h(b) \subseteq -h(a) \cup  h( \sim a) \cup h(b)$.

\

$(\leftarrow)$\ Let $P \in h(a) \rightarrow_{w}h(b)$, then there is a prime filter $P_{0} \in Ob_{1} \cup Ob_{2}$ such that $P\in R(P_{0})$ with $\ms{\textbf{(1)}}$ $R(P_{0}) \subseteq -h(a) \cup h(\sim a) \cup h(b)$. We will show that $P \in h(a \rightarrow_{w}b)$.

\medskip

Assume $b \in P$ or $\sim a \in P$. Since both $\sim a$ and $b$ are $\leq\ \sim a \vee b \leq a \rightarrow_{w}b$ (\cite{Bri-Mont67}, $(N8)$, p.281)  and $P$ is a filter, we deduce, by $(f_3)$, that $a \rightarrow_{w}b \in P$, that is $P \in h(a \rightarrow_{w}b)$.

\medskip

If $b \not \in P$ and $\ms{\textbf{(2)}}$ $\sim a \not \in P$, then by $\ms{\textbf{(1)}}$, $P \in -h(a)$, i.e.\ $a \not \in P$. Let $D$ be the deductive system generated by $P$ and $a$.

On account of (\cite{Ras74}, (4.8), p.94) we consider two cases:

\noindent If $P \in Ob_1$ then $D = \{u \in A : (a \wedge p) \rightarrow_{w}u = 1, \text{with}\ p \in P\}$ (which is equivalent to: $D = \{u \in A : a \rightarrow_{w} u \in P\}$ (\cite{Mont63b}, p.5)). 

\noindent If $P \in Ob_2$ we have $D = \{u :\ \sim u \rightarrow_{w}\ \sim (a \wedge p) = 1, \text{with}\ p \in P\}$ (\cite{Ras74}, p.94).

\

Assume $D = \{u \in A : a \rightarrow_{w} u \in P\}$.     
If $b \in D$ we get $a \rightarrow_{w} b \in P$ so $P \in h(a \rightarrow_{w}b)$.

We willl show that the case $b \not \in D$ leads to a contradiction. In fact, if $b \not \in D$ then $D$ is proper. Therefore, there is an irreducible (= prime) deductive system $C$ such that $P \subseteq D \subseteq C$ and $b \not \in C$. 

Suppose $a \rightarrow_{w}b \not \in C$. Since $\sim a \leq (a \rightarrow_{w}b)$ (\cite{Ras74}, (34), p.70) and $b \leq (a  \rightarrow_{w}b)$ (\cite{Bri-Mont67}, $(N8)$, p.281), we get $\sim a \not \in C$ and $b \not \in C$. Hence $P_{0} \subseteq C$, $a \in D \subseteq C$, $\sim a \not \in C$, and $b \not \in C$. This means that the prime filter $C \in R(P_{0})$ and $C \not \in -h(a) \cup h(\sim a) \cup h(b)$, a contradiction. Hence $a \rightarrow_{w}b \in C$, i.e.\ $a \wedge (a \rightarrow_{w}b) = a \wedge (\sim a \vee b) \in C$. This implies that $\sim a \vee b \in C$, again a contradiction because $C$ is prime, $\sim a \not \in C$ and $b \not \in C$.



\medskip

Suppose now that $D = \{u :\ \sim u \rightarrow_{w}\ \sim (a \wedge p) = 1, \text{with}\ p \in P\}$. Since $a \not \in P$ we deduce $\sim a \not \in\ \sim P$. Thus $\sim a \in -(\sim P) = \varphi(P) \subseteq P$ since $P \in Ob_{2}$. Hence, $\sim a \in P$, in contradiction with the assumption $\ms{\textbf{(2)}}$ ($\sim a \not \in P$).

\medskip

This completes the proof of the theorem \ref{repr-Nelson}.
\end{proof}

\section{\hspace{-2.6ex}.{\hspace{1.ex}}{Conclusion}}

Motivated by an applied problem --the approximation of a set by a pair of sets, called the lower and upper approximation-- and inspired by the Halmos-Monteiro research about monadic Boolean algebras, we have pointed out the `concrete' Rauszer Boolean algebra defined via a preorder $R$ and enriched with an abstract binary operation. 

The $H$-$B$ algebra of $R$-open or $R$-closed sets of this `concrete' algebra gives a general framework to embed some known structures in a unifying way and in a simple and complete presentation. For the sake of illustration, we have considered the representations of $H$-$B$-algebras as well as ${\L}$ukasiewicz, De Morgan, symmetrical Heyting algebras and Nelson ones.

\newpage


\noindent \small{Villeurbanne, France, January 2014}
\[
----------0----------
\]
\end{document}